\documentclass[12pt]{amsart}
\usepackage[utf8]{inputenc}
\usepackage{amsmath,amsthm,amssymb, mathabx}
\usepackage{xcolor}
\usepackage[shortlabels]{enumitem}
\usepackage{hyperref}

\usepackage{graphicx}
\usepackage{amssymb}
\usepackage{tikz}

\newtheorem{theorem}{Theorem}[section]

\newtheorem{remark}[theorem]{Remark}
\newtheorem{lemma}[theorem]{Lemma}

\newtheorem{ques}[theorem]{Question}

\title{A distinction between the paraboloid and the sphere in weighted restriction}
\author{Alex Iosevich and Ruixiang Zhang}

\newcommand{\R}{\mathbb{R}}
\newcommand{\C}{\mathbb{C}}

\newcommand{\Z}{\mathbb{Z}}

\newcommand{\cN}{\mathcal{N}}

\newcommand{\cF}{\mathcal{F}}

\newcommand{\cM}{\mathcal{M}}

\newcommand{\bi}{\mathrm{i}}

\newcommand{\dd}{\mathrm{d}}
\newcommand{\e}{\varepsilon}

\begin{document}

\maketitle

\begin{abstract} For several weights based on lattice point constructions in $\R^d (d \geq 2)$, we prove that the sharp $L^2$ weighted restriction inequality for the sphere is very different than the corresponding result for the paraboloid. The proof uses Poisson summation, linear algebra, and lattice counting. We conjecture that the $L^2$ weighted restriction is generally better for the circle for a wide variety of general sparse weights.
\end{abstract}

\tableofcontents

\section{Introduction}


In this paper, we present some results in \emph{weighted (Fourier) restriction theory}. They concern the size of the best constant $C(X)$ in the estimate
\begin{equation}
    \|\widecheck{f\dd \sigma}\|_{l^p (X)} \leq C(X) \|f\|_{L^q (S, \dd \sigma)},
\end{equation}
where $S \subset \R^d$ is a compact hypersurface with surface measure $\dd \sigma$ and $X \subset \R^d$ is a finite set. \footnote{There are also equivalent formulations where $X$ is of positive Lebesgue measure, in which case the norm $L^p(X)$ is considered instead. See also Remark \ref{blurrem}.} 

Fourier restriction was first studied by Stein \cite{stein1979some} in the ``unweighted'' case where $X$ is the set of all lattice points.\footnote{This version is slightly different from the standard version but both are equivalent by the uncertainty principle. We use this formulation because it fits better into the framework we use in this paper.} He observed that if $S$ has non-vanishing Gaussian curvature, then the behavior of $C(X)$ is often non-trivial. He made a far-reaching conjecture in \cite{stein1979some} called the \emph{(Fourier) Restriction Conjecture} concerning the case where $C(X)$ is a constant.

More precisely, an equivalent formulation of the Restriction Conjecture says that if $S$ is the unit sphere with surface measure $\dd \sigma$, then
\begin{equation} \label{continuousrestriction} \|\widecheck{f\dd \sigma}\|_{l^p (\Z^d)} \lesssim_{p, q} \|f\|_{L^{q} (\dd \sigma)} \end{equation} whenever 
$$ p>\frac{2d}{d-1}, \ \text{and} \ q' \leq \frac{d-1}{d+1}p.$$
This conjecture is still open in dimensions $3$ and higher. See e.g. \cite{St93}; for a thorough description of the problem, and \cite{guth2016restriction, guth2018restriction, hickman2019improved, Wang2022, hickman2023note, wang2022improved, guo2022dichotomy} and the references therein for some recent developments.

In contrast, when $X$ is allowed to be general, the problem is usually more difficult, and even in the case $p=q=2$ it is still far from being well understood.

In both Fourier restriction and weighted restriction, it is the curvature of $S$ that makes the problem highly nontrivial. For natural curved objects in $\R^d$, the sphere $S^{d-1}$ and the (truncated) unit paraboloid $P^{d-1}$ are the first to come to one's mind. There have been many old and new results on Fourier restriction and weighted restriction, and all of them, as far as we are aware, work equally well for the case $S=S^{d-1}$ and the case $S = P^{d-1}$.

The above-mentioned progress has led to applications for many analysis problems, but leaves the following question unaddressed: Does $C(X)$ always have to be essentially the same for the case $S = S^{d-1}$ and the case $S = P^{d-1}$?  In this paper, we exhibit a situation where the answer is no to this question. We construct an $X$ where the best constant $C(X)$ differs significantly between the case of the sphere and the paraboloid. 

It is known that $C(X)$ is the same for $S^{d-1}$ and $P^{d-1}$ for certain $X$. One important such example is the ``train track'' that plays a crucial role in the study of Falconer's Distance Conjecture (see also \S \ref{Falcconj}). See \cite{guth2020falconer} and the references therein for a nice introduction to this example and more background. However, when $X$ is very spread out, it was certainly anticipated by experts that $C(X)$ should be smaller for the sphere. This is because sharp examples (such as the one explained in \S 4.2 of \cite{du2023free}) that work for $P^{d-1}$ usually do not work for $S^{d-1}$. Therefore, $C(X)$ is expected to be much smaller for $S = S^{d-1}$. However, it is very hard to take advantage of $S$ being $S^{d-1}$ using classical harmonic analytic  techniques, since many tools there, including most recent ones, do not see the difference between $S^{d-1}$ and a general positively curved $C^2$ hypersurface. Indeed, we were not aware of any prior examples of $X$ where $C(X)$ is known to be significantly different for $S^{d-1}$ and $P^{d-1}$ in the literature. 

On the other hand, it is known in the context of Stein's original $L^2$ Fourier Restriction Problem that there is no distinction between the sphere $S^{d-1}$ and the paraboloid $P^{d-1}$, $d \ge 2$. Take $q=2$ for example, for smooth surfaces, it is not difficult to show that the sharp $L^2$ restriction exponent $\frac{2(d+1)}{d-1}$ is achieved if and only if the surface has non-vanishing Gaussian curvature. See, for example, \cite{Iosevich99}. For general exponents, the problem is conjectured to behave in the same way for the sphere and the paraboloid too.

\subsection{Main results} We start by stating our main weighted restriction results in two dimensions. Given a compact smooth curve $C \subset \R^2$ with arc length $\dd s$ and an $L^1$ function supported on $C$, define the \emph{Fourier extension} of $f$ to be 
\begin{equation}\label{defnEC}
E_C f (x) = \int_C f(\xi) e^{2\pi \bi x \cdot \xi}\dd s (\xi).
\end{equation}

Let $\beta\geq 0$ be a parameter, and $R>1$ be a large number, consider the following subset:
$$X = X_{\beta} = \left\{(m_1 R^{\beta}, m_2 R^{2\beta}): m_1, m_2 \subset {\mathbb Z}\right\} \cap [-R , R]^2 \subset {\mathbb R}^2.$$ 

Let $P^1$ be the (truncated) unit parabola
$$P^1 = \{(x, x^2): |x| \leq 1\}$$ with length measure $\dd s_P$. By a construction in Barcel\'{o}--Bennett--Carbery--Ruiz--Vilela \cite{barcelo2007some} (see \cite{du2023free} for an exposition), one can deduce the following bound. 

\begin{theorem}\label{paraXlower}
When $0\leq  \beta \leq  \frac{1}{2}$,
\begin{equation}\label{para2D}
\|E_{P^1}\|_{L^2 (\dd s_P) \to l^2 (X)} \gtrsim R^{\frac{1}{2}-\beta}.
\end{equation}
\end{theorem}

Our main result highlights the difference between the circle and a parabola in this context. Let $S^1$ be the unit circle $\subset \R^2$ with length measure $\dd s_C$. We will prove the following weighted restriction theorem:

\begin{theorem}\label{circlerestr}
When $0\leq \beta \leq \frac{1}{28}$,
$$\|E_{S^1}\|_{L^2 (\dd s_C) \to l^2 (X)} \lesssim_{\e} R^{\frac{1}{2}-\frac{3}{2}\beta + \e}.$$ 

\vskip.125in 

Equivalently, 

\begin{equation}\label{main2D}
\|E_{S^1} f\|_{l^2 (X)} \lesssim_{\e} R^{\frac{1}{2}-\frac{3}{2}\beta + \e} \|f\|_{L^2 (\dd s_C)}, \end{equation}  
\end{theorem}

Theorem \ref{circlerestr} and a variant (Theorem \ref{circlesqrestr}) will be proved in \S \ref{pf2Dsec}. We will then prove a high dimensional generalization (Theorem \ref{sphererestr}) in Section \ref{highDsection}.

\subsection{Summary of our approach} We briefly summarize the main ideas of the proof. The set $X$ is chosen to be a lattice inside a large ball and it turns out that we can prove a key relation \eqref{normiscounteq} using the Poisson summation formula. It connects $C(X)$ with a lattice point count whenever $X$ is the set of lattice points inside a box. This allows us to take advantage of the fact that the lattice counting numerology is very different for sufficiently small neighborhoods of $S^1$ (or $S^{d-1}$) and $P^1$ (or $P^{d-1}$). When we take a circle or an ellipse, the number of lattice points on it is controlled efficiently using a theorem in the work of Heath-Brown (\cite{heath1997density}, Theorem 3). Using linear algebra, we show that when the neighborhood of an ellipse is sufficiently small, all the lattice points in the neighborhood lie on another ellipse. This allows us to apply Heath-Brown's theorem to finish the proof.

It is interesting to note that weighted restriction theory and harmonic analysis in general have been useful in establishing lattice counting bounds, with the Gauss circle problem being perhaps the best-known example. See e.g. \cite{bourgain2016proof} or the very recent works \cite{kiyohara2022lattice, li2023improvement}, with the last one containing an improvement to the Gauss circle problem. A general result about lattice counting in neighborhoods of a class of hypersurfaces is proved   in \cite{IT2011}. We refer the readers to references therein for more background on lattice point counting. One highlight of the present paper is that we are able to make things work in the reverse direction and use lattice counting to obtain new weighted restriction estimates, particularly those that behave differently when the paraboloid is replaced by the sphere.

\emph{Incidence geometry} is an important motivation of the weighted restriction estimates we have been focusing on. It is interesting that in incidence geometry there are also (conjectural) differences between incidence estimates associated with parabolas and circles. We will have a discussion of them in \S \ref{differenceinincidencesec}. They are, in general, conjectural, but we give an example of a restricted domain convolution operator bound where the difference between the circle and a parabola is quite concrete and consistent with the spirit of the main result of this paper.

\subsection{Difference between the circle and parabola in incidence geometry}\label{differenceinincidencesec}

Incidence geometry concerns how certain geometric objects (e.g. lines, circles, etc.) can intersect. It is a folklore heuristic that many incidence geometric problems should behave differently when the objects are changed from unit circles (spheres) to unit parabolas (paraboloids). This dichotomy has a long and rich history and we describe some such problems below. 

\subsubsection{The Erd\H{o}s Unit Distance Conjecture}\label{Erdosdist}

The classical Erd\H{o}s Unit Distance Conjecture \cite{erdos1946sets} 
predicts that if $E \subset {\mathbb R}^2$ consisting of $n$ points, then 
\begin{equation} \label{unitdistanceconjecture} |\{(x,x') \in E \times E: |x-x'|=1\}| \lesssim_{\e} n^{1+\e}, \forall \e > 0. \end{equation} 

The best-known result \cite{spencer1984unit} says that the left-hand side in (\ref{unitdistanceconjecture}) is $O(n^{\frac{4}{3}})$. There has been no progress on this problem since the early 80's. If the spherical distance is replaced by a parabolic distance, it was observed by Pavel Valtr in 2004 (unpublished) that the exponent $\frac{4}{3}$ cannot be improved based on an elementary grid constriction.   Indeed, consider an $N \times N^2$ integer grid centered at the origin, and consider the translates of the parabola $y=x^2$ by every point in the grid. This gives us $N^3$ points and $N^3$ parabolas, while the number of incidences is $\sim N^4$.

While the above-described results do not quite prove that there is a fundamental difference between the parabola and the circle since the single-distance conjecture is not resolved, the evidence is quite strong, and this is just the tip of the proverbial iceberg. 

\subsubsection{Falconer's version for the Unit Distance Conjecture}

Another interesting instance where the paraboloid-sphere distinction arises is in the study of Falconer's continuous version of Erd\H{o}s's distance problems. Falconer (\cite{Falconer1986}) proved that if $\mu$ is a compactly supported Borel measure in ${\mathbb R}^d$, $d \ge 2$, then 
\begin{equation} \label{LinftyFalconer} \mu \times \mu \{(x,y): 1 \leq |x-y| \leq 1+\epsilon \} \leq C \cdot I_{\frac{d+1}{2}}(\mu) \cdot \epsilon,\end{equation} where $I_s(\mu)$ is the usual energy integral 
$$ \int \int {|x-y|}^{-s} d\mu(x)d\mu(y).$$

Mattila (\cite{Mattila1987}) proved that the exponent $\frac{d+1}{2}$ cannot be improved when $d=2$, but the higher dimensional case remained open. Iosevich and Senger (\cite{Iosevich2016}) extended Mattila's example to three dimensions, but in dimensions four and higher they were only able to establish the sharpness of the $\frac{d+1}{2}$ exponent for the parabolic metric, not the Euclidean metric, since the argument crucially relies on the largeness of the number of lattice points on a dilated paraboloid. This theme will be exploited in this paper as well. 

\subsubsection{Falconer's (Distinct) Distance Conjecture}\label{Falcconj}

Erd\H{o}s also has a conjecture on distinct distances. One version of this conjecture predicts $n$ points in $\R^d$ determine $\gtrsim_{\e} n^{\frac{2}{d}-\e}$ distinct distances. The two dimensional case of the above version was proved by Guth-Katz in \cite{guth2015erdHos} and the high dimensional case remains open.

Falconer made a continuous version of this conjecture too, which is commonly known as \emph{Falconer's Distance Conjecture}. In the study of this problem, Du (\cite{D2020}) and Iosevich-Rudnev (\cite{IR07}) proved results, all based on the properties of the integer lattice points in one way or another, showing that the key object in the study of the Falconer distance problem, 
$$ \int_{S^{d-1}} {|\widehat{\mu}(R\omega)|}^2 d\omega,$$ where $\mu$ is a Borel measure supported on a compact subset of ${\mathbb R}^d$, $d \ge 2$, cannot satisfy the best possible estimates if $S^{d-1}$ is replaced by the paraboloid, greatly limiting how close one come to the resolution of the Falconer distance conjecture using current methods. See, for example, \cite{DZ19} and \cite{GIOW20} to put this into the context of the Falconer distance conjecture and related problems. 

\vskip.125in 

An interesting feature of the examples given in this section so far is that while the differences between the case of the sphere and the paraboloid are conjectured, none of them is actually known.

\vskip.125in 

\subsubsection{Szemeredi-Trotter incidence theorem and convolution bounds} The Spencer-Szemeredi-Trotter bound \cite{spencer1984unit} introduced in \S \ref{Erdosdist} is also valid for the parabolic distance setting. Furthermore, it can be generalized to a Szemeredi-Trotter incidence theorem for unit parabolas (or unit circles) by the same proof. One version of this theorem asserts that for $n$ translations for a parabola (or a circle), the number of $k$-rich points on them is $O(\frac{n^2}{k^3}+\frac{n}{k})$. It can be transferred to restricted weak-type inequalities for convolution operators as we explain below.

Let $\Gamma_R=\{(n,n^2):n \in \Z, |n|\leq R\}$ where $R>1$. For a compactly supported $f$ on $\Z^2$, 
$$ A_{\Gamma_R}f(x)=\sum_{y \in \Z^2: x-y \in \Gamma_R} f(y).$$
It follows that if $E \subset \Z^2$ of size $>R^2$, then  
$$|\left\{x \in \Z^2: A_{\Gamma_R} 1_E(x)>k \right\}| \lesssim {|E|}^2 k^{-3},$$ which can be viewed as a $\left(\frac{3}{2}, 3 \right)$ restricted weak-type bound for $A_{\Gamma_R}$. This bound is sharp for e.g. $k=\frac{R}{2}$ and $E=$ the $R \times R^2$ grid centered at the origin.


On the other hand, if $\Gamma_R$ is replaced by the set of integer points on a circle or an ellipse at scale $R$ or $R^2$. Then by lattice counting,  the total number of lattice points is $O_{\e} (R^{\e})$. From this  one can trivially get much better bound for the new convolution operator.

\section*{Acknowledgements}

The first listed author was supported in part by the National Science Foundation under grant no. HDR TRIPODS - 1934962 and by the NSF DMS - 2154232. The second listed author was supported in part by NSF DMS-2143989 and the Sloan Research Fellowship.  The second listed author would like to thank Larry Guth who introduced him to Question \ref{onepointques}, and Xiumin Du and Yumeng Ou for helpful discussions on weighted restriction. 

\section*{Notations}

We will use $\e$ to denote an arbitrarily small positive number. $C$ or $C(d, k)$ etc. will denote constants depending only on specified parameters that may vary from line to line. For a set $X$ in an Euclidean space and $\delta>0$, denote the $\delta$-neighborhood of $X$ by $\cN_{\delta} (X)$.

Let $\gamma>0$ be a real parameter. We say that $X$ is $O(Y)$ or $X \lesssim Y$ if there exists a positive constant $C$  such that $X \leq C Y$. We say $X$ is $O_{\gamma} (Y)$ or $X \lesssim_{\gamma} Y$ if the $C$ above depends on $\gamma$. In the above situations, we say $Y = \Omega(X)$ or $Y=\Omega_{\gamma} (X)$. 

\vskip.25in 

\section{Proof of Theorem \ref{circlerestr}}\label{pf2Dsec}

In this section, we prove Theorem \ref{circlerestr}. The proof relies on Poisson summation and the following counting lemma:

\begin{lemma}\label{counting2D}
    For all $0\leq \beta \leq \frac{1}{28}$ and $R>1$, every translation of $\cN_{10^{-100}R^{-1}} (S^1)$ contains $O_{\e} (R^{\e})$ points in $R^{-\beta} \Z \times R^{-2\beta} \Z$.
\end{lemma}

\begin{remark} We emphasize that the implied constant in Lemma \ref{counting2D} is uniform, independent of the translation.
\end{remark}

We first prove Theorem \ref{circlerestr} assuming Lemma \ref{counting2D}. It has roots in the classical arguments in \cite{bourga1991besicovitch} and \cite{shayya2021fourier} (see \S 6 of \cite{shayya2021fourier} ). The main difference is that we get to exploit the oscillation of the kernel function by taking advantage of the lattice structure and using  Poisson summation, leading to a better bound for the circle.

\begin{proof}[Proof of Theorem \ref{circlerestr} assuming Lemma \ref{counting2D}]
Let $M$ be the operator norm $\|E_{S^1}\|_{L^2 (\dd s_C) \to l^2 (X)}$. Let $L$ be the lattice $R^{\beta} \Z \times R^{2\beta}\Z$ and $L^* = R^{-\beta} \Z \times R^{-2\beta}\Z$ be its dual. Fix a smooth bump function $\psi$ supported inside the unit disk such that $\psi, \hat{\psi} \geq 0$ and $\psi (0)>0$.

\vskip.125in 

Recall the definition of $E_{C}$ in \eqref{defnEC}, we have the adjoint
$$E_{S^1}^* g (\xi) = \int_{{\mathbb R}^2} g(x) e^{-2\pi \bi x \cdot \xi}\dd x.$$

Hence the kernel of $E_{S^1} E_{S^1}^*$ is
$$K_{E_{S^1} E_{S^1}^*}(x_1, x_2) = \int_C e^{2\pi \bi (x_1-x_2) \cdot \xi} \dd s_C(\xi) = \widecheck{\dd s_C} (x_1-x_2).$$

The $TT^*$ method yields
\begin{eqnarray}\label{TTstar}
M^2 = & \max_{\left\{f: X \to \C : \|f\|_{l^2} = 1 \right\}} \sum_{x_1, x_2 \in X} f(x_1) \overline{f (x_2)} \widecheck{\dd s_C} (x_1-x_2)\nonumber\\
\lesssim & \max_{\left\{f: L \to \C : \|f\|_{l^2} = 1 \right\}} \sum_{x_1, x_2 \in X} f(x_1) \overline{f (x_2)} \widecheck{\dd s_C} (x_1-x_2)\psi (\frac{x_1 - x_2}{R}).
\end{eqnarray}

\vskip.125in 

Note the right-hand side of \eqref{TTstar} is the operator norm of the convolution operator with the kernel $K = \widecheck{\dd s_C} (\cdot)\psi (\frac{\cdot}{R})$ on $l^2 (L)$. This is equal to the $L^{\infty}$ norm of the Fourier transform\footnote{The ``Fourier transform'' $\cF$ here should be understood as the process of constructing a periodic function from its Fourier series. We use a different notation to distinguish it from the Euclidean Fourier transform, defined by $\hat{\cdot}$.} $\cF K$ of $K|_L$ on $\R^2 / L^*$. Take an arbitrary $\xi_0 \in {\mathbb R}^2$ and observe that by Poisson summation, 
\begin{eqnarray}\label{PSFK}
\cF K (\xi) = & \sum_{x \in L} K(x) e^{2 \pi \bi x \cdot \xi_0}\nonumber\\
= & R^{-3\beta} \sum_{\xi \in L^*} \hat{K} (\xi - \xi_0).
\end{eqnarray}

We observe that $\hat{K}$ is the convolution of $\dd s_C$ with the mollifier $R^2 \hat{\psi} (R\cdot)$. Hence by the rapid decaying property of $\hat{\psi}$, the inner sum of the right-hand side of \eqref{PSFK} is $\lesssim R$ times the maximum number of points in $L^*$ inside a translation of $\cN_{\frac{C}{R}}(S^{1})$. By Lemma \ref{counting2D}, the right hand side of  \eqref{PSFK} is $O_{\e} (R^{1-3\beta + \e})$.

\end{proof}

\begin{remark}\label{strongpfrem}
One easily sees that the same $\lesssim$ in the third last line of the last proof can be replaced by $\gtrsim$ if we choose a slightly smaller $C$ in the second last line. A similar statement can be made about the $\lesssim$ in \eqref{TTstar} since the support of the convolution kernel $K$ is on scale $R$. Thus the above proof is strong in the sense that it is an equality connecting $M$ and some point counting data. More precisely by the above argument, we get
\begin{equation}\label{normiscounteq}
    M^2 \sim R^{1-3\beta} \max_{U=\text{ a translation of }\cN{\frac{1}{R}}(S^1)} |U \cap L^*|.
\end{equation}
\end{remark}

\begin{remark}\label{blurrem}
By Minkowski's inequality, the conclusion of Theorem \ref{circlerestr} does not change if one replaces $X=X_{\beta}$ by its $1$-neighborhood and $l^2$ by $L^2$. This ``blurred out'' version has a sometimes more familiar setting in the weighted restriction theory literature. From the uncertainty principle, we also know this ``blurred out'' version is morally equivalent to the original Theorem \ref{circlerestr}.
\end{remark}

We now prove Lemma \ref{counting2D}. The key observation is that when $\beta$ is small enough, all the lattice points in the neighborhood must lie on an ellipse. One can then use a counting lemma in the work of Heath-Brown \cite{heath1997density}.

The case $\beta=0$ is trivial and we assume $\beta>0$ henceforth. First note that by rescaling and taking $r=R^{\beta}$, Lemma \ref{counting2D} is equivalent to the following lemma:

\begin{lemma}\label{countingthm}
    For $r>1$ and all $0<\beta \leq \frac{1}{28}$, every translation of the Minkowski sum 
    $$Ell_{\beta, r}+ [-10^{-100}r^{-\frac{1-\beta}{\beta}}, 10^{-100}r^{-\frac{1-\beta}{\beta}}] \times [-10^{-100}r^{-\frac{1-2\beta}{\beta}}, 10^{-100}r^{-\frac{1-2\beta}{\beta}}]$$ contains $O_{\e} (r^{\e})$ integer points.
    
    Here the ellipse $Ell_{\beta, r}$ is given by the equation $$\left(\frac{x}{r}\right)^2 + \left(\frac{y}{r^2}\right)^2 = 1.$$
\end{lemma}

Next we prove Lemma \ref{countingthm}, and hence Lemma \ref{counting2D}. 



\begin{proof}[Proof of Lemma \ref{countingthm}]
    It suffices to prove the $\beta = \frac{1}{28}$ case, where $\frac{1-\beta}{\beta} = 27$. Let $S$ be the set of integer points in the Lemma. We use $\tilde{Ell}$ to denote the translation of $Ell_{\beta, r}$ in the assumption. 
    Without loss of generality, we assume  $r \geq 10$ and that both components of the translation do not exceed $1$ in absolute value.  

    First, we claim that all $(x, y) \in S$ must be on the same quadratic curve. This is equivalent to asserting all $(x, y) \in S$ make the vector $v = v(x, y) = (1, x, y, x^2, y^2, xy)$ in one same hyperplane in $\Bbb{R}^6$, in turn equivalent to vanishing of all $6  \times 6$ determinants formed by $6$ such vectors constructed from any $6$ points in $S$. Take any such determinant, note that the $j-th$ row $v_j$ can be written as $v_j = \mu_j + \nu_j$ ($1 \leq j \leq 6$) satisfying: 
    
    (i) $\mu_j$ is similarly constructed from a nearby point on $\tilde{Ell}$. That the points are all on $\tilde{Ell}$ shows that all $\mu_j$ are on a same hyperplane. Moreover, its entries have size $(1, O(r), O(r^2), O(r^2), O(r^4), O(r^3))$ with the implicit constant all bounded by $2$ (note we have taken the translation into account).

    (ii) By the mean value theorem, the ``error term'' $\nu_j$ has its components bounded by 
    $$0, 10^{-100}r^{-27}, 10^{-100}r^{-26}, 10^{-99} r^{-26}, 10^{-99} r^{-24}, 10^{-99}r^{-25},$$ respectively.

    By (i) and (ii), looking at all possible summands in the determinant with at least one error term we conclude that the $|6\times 6$ determinant$|$ is $\leq 10^{-50}r^{-16}$. But it must be an integer and hence is $0$.

\textbf{Claim.} There is a universal constant $C>0$ such that there exists a vector $q \perp$ all above $v (x, y)$ in $\Bbb{R}^6$  and that $q$ has the last $3$ components $10^{-10}$-close to $(r^2, 1, 0)$ and with all entries being rationals with height $O(r^{C})$.

\begin{proof}[Proof of \textbf{Claim}]
Assume all possible $v = v(x, y)$ span a subspace $V \subset \R^6$ of rank $k$. Note $k\leq 5$ by the last paragraph.\footnote{In fact by B\'{e}zout we can deduce $k=5$ as long as $|S|\geq 5$. But for the sake of higher dimensional generalizations in \S \ref{highDsection}, we use a more general argument here not needing this information.} We can choose $k$ linearly independent vectors $v$, and will still use $v_1, \ldots, v_k$ to denote them in this paragraph for simplicity. 
We continue to use the decomposition $v_j = \mu_j + \nu_j$ as earlier in the proof. 

By the theory of linear systems, any point in $V^{\perp}$ can be written as $T (a_1,\ldots, a_{6-k})$ where $T$ is a linear isomorphism: $\R^{6-k} \stackrel{\simeq}{\to} V^{\perp} \subset \R^6$. Furthermore, we can assume $6-k$ coordinates among the $6$ components of  $T$ are coordinate functions of $\R^{6-k}$. 

We use $\tilde{V}$ to denote the span of $\mu_1, \ldots, \mu_k$. By analyzing the error similarly as before we conclude $\tilde{V}$ has dimension $k$. Set up an $\tilde{T}: \R^{6-k} \stackrel{\simeq}{\to} \tilde{V}^{\perp}$ similarly as we did for $T$ and we can assume $\tilde{T}$ has $6-k$ components being coordinate functions, with the same position as in $T$. The entries of the matrix $\cM (T)$ of $T$ and the matrix $\cM (\tilde{T})$ of $\tilde{T}$ are given by Cramer's rule. Analyzing the errors same as before, we see that the difference between each entry of $\cM (T)$ and the same entry of $\cM (\tilde{T})$ is bounded by $10^{-40}r^{-4}$. This can be seen by noticing the following: An entry of $\cM (T)$ has the form $A/B$ with $A$ an integer $\leq 10^{10} r^{12}$ and $B$ a nonzero integer, and the same entry of  $\cM (\tilde{T})$  has the form $(A+Error1)/(B+Error2)$ with $|Error1|, |Error2| \leq 10^{-50}r^{-16}$.

Suppose the equation of $\tilde{Ell}$ is $$r^2 x^2 + y^2 +  c_2 x + c_3 y + c_1 = 0$$ with all $|c_j| \leq 10r^4$. Then we have found a vector $w = (c_1, c_2, c_3, r^2, 1, 0) \in \tilde{V}^{\perp}$ with $|w|\leq 100r^4$. Suppose $w = \tilde{T}\alpha, \alpha \in \R^{6-k}$. Then $|\alpha| \leq 100r^4$ since each component of $\alpha$ is some component of $w$ by the setup of $\tilde{T}$. Let $C_1> 1$ large to be determined and find an approximation of $\alpha$ by a rational vector $\alpha_1$ of height $\leq r^{C_1 + 10}$ such that $$|\alpha-\alpha_1| \leq 10 r^{-C_1}.$$

Now consider $q = T\alpha_1$. It is a rational vector of height $O(r^{C_1+1000})$ and is in $V^{\perp}$. Moreover,
\begin{eqnarray}
      |q-w| = &|T\alpha_1 -\tilde{T}\alpha|\nonumber\\
      \leq & |T(\alpha_1-\alpha)|+|(T-\tilde{T})\alpha|\nonumber\\
      \leq & \|T\|\cdot |\alpha_1-\alpha|   + \|T-\tilde{T}\|\cdot |\alpha| \nonumber\\
      \leq & 10^{20} r^{12} \cdot 10 r^{-C_1} + 10^{-30} r^{-4} \cdot 100r^4  \nonumber\\
      \leq & 10^{-10}
\end{eqnarray}
for sufficiently large $C_1$. Take this $q$ and we finish the proof of \textbf{Claim}.
\end{proof}

The \textbf{Claim} asserts that all points on $V$ are orthogonal to $q=(q_1, q_2, \ldots, q_6)$. Hence every integer point in $S$ is on the quadratic curve $$Q: q_4 x^2 + q_5 y^2 + q_6 xy + q_2 x + q_3 y + q_1 = 0.$$

Moreover, all $|q_j|$ has height $O(r^C)$ and $q_4, q_5, q_6$ are $10^{-10}$-close to $r^2, 1, 0$, respectively. The last property implies the quadratic part of $Q$ is non-degenerate. We now apply Theorem 3 in \cite{heath1997density} to $Q$ and conclude the total number of such integer points is $O_{\e}(r^{\e})$.
\end{proof}

\begin{remark}
   The upper bound of $\beta$ in Lemma \ref{countingthm} is worked out explicitly here and probably can be improved a bit. Easy improvements will make the proof somewhat messier and is probably far from optimal anyway, so we keep the current proof. Also, the particular form of $E$ is not important in the above lemma and the same proof works for $E$ replaced by e.g. a circle of radius $r$. For example, the above proof gives the following variant of Theorem \ref{circlerestr}.
\end{remark}

\begin{theorem}\label{circlesqrestr}
Let $\tilde{X} = R^{\beta}\Z \times R^{\beta}\Z \cap [-R, R]^2 \subset \R^2$. Then for all sufficiently small $\beta \geq 0$,
\begin{equation}\label{mainsq2D}
\|E_{S^1}\|_{L^2 (\dd s_C) \to l^2 (\tilde{X})} \lesssim_{\e} R^{\frac{1}{2}-\beta + \e}. \end{equation}  
\end{theorem}

Theorem \ref{circlesqrestr} and its variants can create many more interesting distinctions between the circle and the parabola. Since the proof of Theorem \ref{circlerestr} gives essentially an equality (see Remark \ref{strongpfrem} and the relation \eqref{normiscounteq} there), we see  by the same proof and lattice point counting for the parabola that if $R^{\frac{\beta}{2}}$ is an integer and if we replace $S^1$ by $P^1$, we have for all $0<\beta < 1$ that
\begin{equation}\label{parasqeq}
\|E_{P^1}\|_{L^2 (\dd s_C) \to l^2 (\tilde{X})} \gtrsim R^{\frac{1}{2}-\frac{3}{4}\beta}.
\end{equation}

Note that \eqref{parasqeq} can also be seen directly by looking at the Barcel\'{o}--Bennett--Carbery--Ruiz--Vilela example used for Theorem \ref{paraXlower} with $\beta$ replaced by $\beta/2$ and viewing $\tilde{X}$ as a sublattice.

\begin{remark}
    We remark that there are also regimes where we can confirm there is no distinction between the circle and the parabola for the weighted restriction theory of $\tilde{X}$. For example, when $\frac{3}{4} \leq \beta< 1$, we run the $TT^*$ method in the proofs of Theorems \ref{circlerestr} and \ref{circlesqrestr} and see the convolution kernel has bounded $L^1$ norm by stationary phase. Hence in this regime we have $$\|E_{S^1}\|_{L^2 (\dd s_C) \to l^2 (\tilde{X})} \sim \|E_{P^1}\|_{L^2 (\dd s_C) \to l^2 (\tilde{X})} \sim 1.$$

    It will be very interesting to prove the existence or the non-existence of a distinction for larger $\beta$ below $\frac{3}{4}$.
\end{remark}




\section{Higher dimensional results}\label{highDsection}

One can generalize the proofs of Theorems \ref{paraXlower} and \ref{circlerestr} to $\R^d$ in the same way. Let $P^{d-1}$ and $S^{d-1} \subset \R^d$ be the truncated unit paraboloid and the unit sphere, respectively, and define their Fourier extension operators $E_{P^{d-1}}$ and $E_{S^{d-1}}$ using the hypersurface volume measure $\dd \sigma_{S^{d-1}}$ and $\dd \sigma_{P^{d-1}}$ similar to the definition \eqref{defnEC}. For $0 < \beta <\frac{1}{2}$, let $$X_d = R^{\beta}\Z\times \cdots \times R^{\beta}\Z \times R^{2\beta}\Z \cap [-R, R]^d.$$ The Barcel\'{o}--Bennett--Carbery--Ruiz--Vilela example is still valid and gives

\begin{theorem}\label{paradXlower}
When $0< \beta < \frac{1}{2}$,
\begin{equation}\label{paradD}
\|E_{P^{d-1}}\|_{L^2 (\dd \sigma_{P^{d-1}}) \to l^2 (X_d)} \gtrsim R^{\frac{1}{2}-\beta}.
\end{equation}
\end{theorem}

\begin{remark}
    Theorem \ref{paradXlower} can also be seen by the direct generalization of \eqref{normiscounteq}. See also the proof of Theorem \ref{sphererestr} below.
\end{remark}

On the other hand, for $S^{d-1}$ we have the following analogue of Theorem \ref{circlerestr}:

\begin{theorem}\label{sphererestr}
When $0\leq \beta \leq \Omega(\frac{1}{d^2})$,
\begin{equation}\label{maindD}
\|E_{S^{d-1}}\|_{L^2 (\dd \sigma_{S^{n-1}}) \to l^2 (X_d)} \lesssim_{\e} R^{\frac{1}{2}-\frac{3}{2}\beta + \e}.
\end{equation} 
\end{theorem}

\begin{proof}[Sketch of the proof of Theorem \ref{sphererestr}]
    The proof will follow exactly the same framework as in the proof of Theorem \ref{circlerestr} and we only provide a sketch here, highlighting how the numerology works.

    We use $M_d$ to denote $\|E_{S^{d-1}}\|_{L^2 (\dd \sigma_{S^{n-1}}) \to l^2 (X_d)}$. In this proof, $L$ will denote $R^{\beta}\Z\times \cdots \times R^{\beta}\Z \times R^{2\beta}\Z \subset \R^d$. Analogous to \eqref{normiscounteq}, we obtain
    \begin{equation}\label{normiscounteqhighd}
    M_d^2 \sim R^{1-(d+1)\beta} \max_{U=\text{ a translation of }\cN{\frac{1}{R}}(S^{d-1})} |U \cap L^*|.
\end{equation}

Hence it suffices to prove that when $0< \beta < \Omega(\frac{1}{d^2})$, the counting on the right hand side is 
\begin{equation}\label{highdcountingbd}
O_{\e} (R^{(d-2)\beta+ \e}).
\end{equation}
This can be deduced directly from Theorem 3 in \cite{heath1997density} as follows:

We do a rescaling and just like how we reduced the two dimensional case to Lemma \ref{countingthm}, we want to bound the number of lattice points in any translation of the Minkowski sum of $$Ell_d: r^2 x_1^2 + \cdots + r^2 x_{d-1}^2 + x_d^2 = r^4$$ and $$\theta\left([-r^{-\frac{1-\beta}{\beta}}, r^{-\frac{1-\beta}{\beta}}]\times \cdots \times [-r^{-\frac{1-\beta}{\beta}}, r^{-\frac{1-\beta}{\beta}}] \times [-r^{-\frac{1-2\beta}{\beta}}, r^{-\frac{1-2\beta}{\beta}}]\right)$$ where $0<\theta<1$ is a small constant at our disposal and $r = R^{\frac{1}{\beta}}>10$ without loss of generality. Think of $r^{-\frac{1-\beta}{\beta}}$ and $r^{-\frac{1-2\beta}{\beta}}$ as ``error'' (denoted by $Error$). We will again find a rational quadratic hypersurface (denoted as $Q$) with height $O(r^C)$ and all coefficients $\theta$-close to that of a translated $Ell_d$. To achieve this, again find a set of integer points under consideration such that the vectors $(x_1, \ldots, x_d, x_1^2, x_1 x_2, \ldots, x_1 x_d, x_2^2, \ldots, x_d^2)$ form a maximal independence set. Consider nearby points on the translation of $Ell_d$ and compare the determinant again, we see the error of the determinant is $O_d (\theta\cdot r^{O(d^2)}\cdot Error)$. As long as this is $<1$, the above vectors span a space that does not have full rank. Now the same procedure as in the proof of Lemma \ref{countingthm} produces the above said quadratic hypersurface $Q$ as long as $O_d(\theta\cdot r^{O(d)}r^{O(d^2)}r^{O(d^2)}\cdot Error) < \theta$ (the first two factors correspond to upper bounds on $|\alpha|$ and $\|T-\tilde{T}\|$ in the language of the proof of Lemma \ref{countingthm}).

Now our problem is reduced to counting integer points on $Q$, all of whose coefficients are $\theta$-close to that of a translation of $Ell_d$. Looking at the spatial position of the ellipsoid $Ell_d$, we see there are $O_d(r)$ possibilities for each of $x_1, \ldots, x_{d-2}$. Freezing a particular choose of $x_1, \ldots, x_{d-2}$, Theorem 3 in \cite{heath1997density} bounds the number of possible lattice points by $O_{d, \e}(r^{\e})$ as long as $\theta$ is very small. By this we get \eqref{highdcountingbd}, as long as $O_d(\theta \cdot r^{O(d)}r^{O(d^2)}r^{O(d^2)}\cdot Error) < \theta$. This can be guaranteed for the range $0< \beta < \Omega(\frac{1}{d^2})$ with the implicit constants and $\theta$ carefully chosen.

\end{proof}




\section{Further directions}

Part of our purpose of studying the weighted restriction theory of objects like $X$ and $\tilde{X}$ is to find new insights in understanding the weighted restriction theory for more general subsets, especially those that are sufficiently spread out. Here is a precise question of interest:

\begin{ques}\label{onepointques}
    Let $Y$ be a set of points in $B_R \subset \R^2$ such that there is at most one $y\in Y$ in any $R^{\frac{1}{2}}$-ball. What is $$\|E_{S^1}\|_{L^2 (\dd s_C) \to l^2 (Y)}$$ and $$\|E_{P^1}\|_{L^2 (\dd s_P) \to l^2 (Y)}?$$
\end{ques}

And in the spirit of this paper, we can ask:
\begin{ques}\label{distinctionques}
    Are the two operator norms in Question \ref{onepointques} different?
\end{ques}

Our $Y$ form a family of typical and interesting examples of sets that are very spread out. We already saw the operator norms are very different for $\tilde{X}$ and perhaps should expect the same for general $Y$. A positive answer to Question \ref{distinctionques} will shed light on the study of many central problems in geometric measure theory and analysis such as certain variants of the Falconer's distance problem. See, for example, \eqref{LinftyFalconer} above. Unfortunately, when $Y$ does not have a lattice structure it seems very difficult to prove the difference.

\vskip.125in 

Another realm where the difference between the circle and a parabola can potentially be exhibited is vector spaces over finite fields. To be more explicit, we fix $d=2$ and consider $$ S=\{x \in {\mathbb F}_q^2: x_1^2+x_2^2=1\},$$ where ${\mathbb F}_q$ is the finite field with $q$ elements.
Let 
$$ P=\{x \in {\mathbb F}_q^2: x_2=x_1^2\}.$$

Fix a non-trivial (additive) character $\chi$  on ${\mathbb F}_q$. Given $f: {\mathbb F}_q^2 \to {\mathbb C}$, define the Fourier transform of $f$ by 
$$\hat{f} (m) = q^{-2} \sum_{x \in {\mathbb F}_q^2} \chi(-x \cdot m) f(x).$$

Identifying $P$ with its indicator function, it is not difficult to check using elementary properties of Gauss sums that 
$\widehat{P}(0,0)=q^{-1}$, $\widehat{P}(m_1,0)=0$ if $m_1 \neq 0$, and $|\widehat{P}(m)|=q^{-\frac{3}{2}}$ if $m_2 \neq 0$.  On the other hand,  It is well known (see e.g. \cite{IR07}) that 
$$ |\widehat{S}(m)| \leq 2q^{-\frac{3}{2}},$$ using the properties of twisted Kloosterman sums, but the actual distribution of values of $\widehat{S}(m)$ is very subtle. See, for example, the work \cite{Katz1988} on Kloosterman angles. This work makes it clear that $|\widehat{S}(m)|$ can be considerably smaller than $2q^{-\frac{3}{2}}$, but how much and how often is a very delicate question.

\begin{ques} \label{finitefields}  

How often can it happen for $$|\widehat{S}(m)| < q^{-1.51},$$ and what can we say about the set of these $m$?
\end{ques} 

There are numerical and heuristic reasons to believe that such an estimate holds for a non-trivial number of values of the input variables. We plan to address this question in detail in the sequel. 



\bibliographystyle{plain}
\bibliography{main}

\end{document}